\documentclass[]{interact}

\usepackage{epstopdf}
\usepackage{subfigure}
\usepackage{xcolor,soul}
\usepackage[numbers,sort&compress]{natbib}
\usepackage{url}
\bibpunct[, ]{[}{]}{,}{n}{,}{,}

\newcommand{\Real}{\mathbb{R}}

\theoremstyle{plain}
\newtheorem{theorem}{Theorem}[section]
\newtheorem{lemma}[theorem]{Lemma}

\theoremstyle{definition}

\theoremstyle{remark}

\begin{document}


\title{A New Spectral Conjugate Subgradient Method with Application in Computed Tomography Image Reconstruction}

\author{
\name{M. Loreto\textsuperscript{a}\thanks{CONTACT M. Loreto. Email: mloreto@uw.edu} and T. Humphries\textsuperscript{a} and C. Raghavan\textsuperscript{b} and K. Wu\textsuperscript{c} and S. Kwak\textsuperscript{d}  }
\affil{\textsuperscript{a} University of Washington Bothell, USA; \textsuperscript{b} New York University, USA; \textsuperscript{c} University of Washington, USA; \textsuperscript{d} University of Oregon, USA.}
}

\maketitle

\begin{abstract}
A new spectral conjugate subgradient method is presented to solve nonsmooth unconstrained optimization problems. The method combines the spectral conjugate gradient method for smooth problems with the spectral subgradient method for nonsmooth problems. We study the effect of two different choices of line search, as well as three formulas for determining the conjugate directions. In addition to numerical experiments with standard nonsmooth test problems, we also apply the method to several image reconstruction problems in computed tomography, using total variation regularization. Performance profiles are used to compare the performance of the algorithm using different line search strategies and conjugate directions to that of the original spectral subgradient method. Our results show that the spectral conjugate subgradient algorithm outperforms the original spectral subgradient method, and that the use of the Polak–Ribi\`ere formula for conjugate directions provides the best and most robust performance.
\end{abstract}


\begin{amscode}90C30, 90C56, 94A08\end{amscode}

\section{Introduction}

{The problem under consideration is the minimization of nonsmooth functions without any constraints
\begin{equation} \label{probm}
  \min_{x\in\Real^{n}}~~ f(x),
   \end{equation}
  where $f: \Real^{n} \rightarrow \Real$ is locally Lipschitz continuous in its domain and differentiable  on an open dense subset of $ \Real^n$. We are interested in the specific case where $f$ is continuously differentiable almost everywhere, but possibly not differentiable at minimizers.}

To solve \eqref{probm}, Loreto {\it et al.} in \cite{Loreto:19} proposed the spectral subgradient method. {It combines a nonmonotone line search~\cite{GriLaLu:86}, a globalization scheme~\cite{wmr2:05}}, and the classical subgradient approach using the the Barzilai-Borwein (spectral) step length~\cite{BarBor:88}. {The significant advantage of the spectral step is that its computation is independent of the optimal function value}. In \cite{Loreto:19}, extensive numerical experimentation was presented, {comparing six classic subgradient methods with the spectral subgradient  with the method of performance profiles~\cite{Dolan02}}. The numerical experiments demonstrated the spectral subgradient method's superiority compared to these methods.

Moreover, Loreto and collaborators have explored combining the spectral step with subgradients and other subdifferential approaches. In \cite{Loreto:07}, they developed the spectral projected subgradient method for the minimization of non-differentiable functions on convex sets. {Successful numerical results were presented for two distinct problem types: set covering and generalized assignment problems, and convergence} results were shown later in \cite{Loreto:15} and \cite{Loreto:16}. In \cite{Loreto:17}, they developed and analyzed the nonsmooth spectral gradient method for unconstrained problems in which the spectral step is combined with subdifferential approaches, such as the simplex gradient and the sampling gradient.

The spectral subgradient method represents an extension to the nonsmooth case of the spectral gradient method developed by Raydan~\cite{Raydan:93,Raydan:97}.  In \cite{Raydan:97}, the author found that the use of gradient descent with spectral step length for unconstrained (smooth) nonlinear optimization provided several advantages over approaches based on the conjugate gradient method for this problem. Inspired by this, Birgin \& Mart\'{\i}nez proposed a spectral conjugate gradient (SCG) method~\cite{Bm01}, which combines the spectral step with conjugate gradient directions. They studied several formulas for determining the conjugate directions (which we describe in Section~\ref{sec:2.1}), and determined that the use of the spectral step provided improved results over the standard conjugate gradient method.

Several modifications to SCG have been proposed subsequently. Andrei proposed the SCALCG algorithm~\cite{andrei07a,andrei07b}, {including a quasi-Newton step to ensure the production of descent directions.} In~\cite{zzl06,dc08}, the authors propose instead modifying the choice of spectral step to guarantee descent, and prove convergence using both Armijo and Wolfe line searches. Subsequent authors have built off these early works by proposing modified formulas for generating search directions with desirable properties, e.g.~\cite{zz08,CL10,andrei2010new,LP12,dwc13,LFZ19}.
Inspired by the advantages of the approximate optimal stepsize approach \cite{LZHL17,LZHL18} and SCG modification to achieve descent directions \cite{Yu08}, Wang {\it et al.}~\cite{wang2020new} presented a new strategy to choose the spectral and conjugate parameters, with the advantage of producing sufficient descent directions and not requiring more cost per iteration.

It is well known the spectral gradient methods (BB methods) perform better than the classical steepest descent method. In \cite{drs056}, the authors aimed to improve the poor practical behavior of the steepest descent method (SD) by analyzing the convex quadratic case and using the second-order information provided by the step to accelerate the convergence of the method, achieving computational results comparable to the BB methods. Many other improvements to the BB methods have been proposed, such as accelerating  the convergence of the BB methods by adding some monotone steps and guaranteeing finite termination \cite{HDLXZ22}.

In addition to work on SCG-type methods, the BB method has also been used and modified to solve many constrained optimization problems, such as image deblurring, compressed sensing, and other inverse problems (\cite{bonettini2008scaled, FMRW07} and references therein), and applications coming from distributed machine learning and multiperiod portfolio optimization~\cite{CDVM23}, among others. Moreover, in the case of a feasible region with single linear equality and lower and open bounds, Crisci {\it et al.}~\cite{CPFVZ20} analyzed how these constraints influence the spectral properties of the BB rules, generating new BB approaches capturing second-order information and exploiting the nature of the feasible region. Furthermore,  a deep analysis of the spectral properties of the BB rules is presented in \cite{Serena23}, when the rules are adopted with the proximal gradient method, and comparisons with well-known state-of-the-art methods are included.


Nearly all of the previously studied extensions to SCG assume {$f$ is continuously differentiable}. It is interesting to note that an extension of the original conjugate gradient method to minimize nondifferentiable functions using subgradients was first presented many years ago by Wolfe~\cite{wolfe74,wolfe75}.  In the interim, however, there appears to have been little work done in the area of using conjugate subgradients for nonsmooth problems, despite some recent papers~\cite{NT14,KSM19,RAR23}.


{In this paper, we introduce the spectral conjugate subgradient (SCS) method as a novel approach for solving nonsmooth problems. The inspiration for this method stems from the SCG approach. Our primary objective is to expand the application of spectral methods to nonsmooth problems, given the success of smooth spectral methods in solving a wide range of practical problems (Birgin et al. \cite{Bmr09, Bmr14})} We note that an extension of spectral CG methods to nonsmooth problems was recently considered in~\cite{RAR23}, who develop a spectral subgradient method for solving the absolute value equation (AVE). Our work experiments with different choices of search direction, and also consider a broader class of problems, including an applied problem in tomographic imaging.

{The organization of this paper is as follows: Section~\ref{S:methodology} provides a description of the SCS method, while Section~\ref{S:CT} introduces an application of interest which involves the use of the nonsmooth total variation (TV) function as a regularizer for several image reconstruction problems in computed tomography (CT). In Section~\ref{S:convergence} we provide some comments on the convergence of the SCS algorithm. Section~\ref{S:experiments} presents numerical results based on a set of nonsmooth problems, including the CT problem, using performance profiles. Finally, Section~\ref{S:conclusions} offers some concluding remarks}. A general proof of convergence for subgradient descent algorithms using nonmonotone line search is provided in Appendix~\ref{A:proof}.

\section{Methodology}
\label{S:methodology}

\subsection{Spectral Conjugate Gradient Method}\label{sec:2.1}

{The original spectral CG (SCG) method for unconstrained optimization was proposed by Birgin \& Mart\'{\i}nez \cite{Bm01}, under the assumption that the objective function $f$ is continuously differentiable at all points.} Letting $g_k$ denote the gradient at the $k$th iterate $x_k$, the vectors $s_k=x_{k+1} -x_{k}$ and $y_{k}=g_{k+1}-g_k$ are defined for $k=0,1,2,\dots$. Beginning from an initial guess $x_0$, the subsequent iterates are given by
\[x_{k+1}= x_k + \alpha_kd_k,\]
where $\alpha_k$ is selected using a Wolfe line search (described in the next section), and the conjugate directions $d_k\in \Real^n$ are generated by
\begin{equation}
d_{k+1} = -\theta_kg_{k+1} +\beta_ks_k. \label{E:searchdir}
\end{equation}

{The spectral step parameter, denoted by $\theta_{k}$, was initially proposed by Barzilai and Borwein \cite{BarBor:88} and subsequently refined by Raydan~\cite{Raydan:93,Raydan:97}:}

 \begin{equation}
   \label{formula}
  \theta_{k} = \frac{s_{k}^T s_{k}}{s_{k}^T y_{k}},
   \end{equation}
with $\theta_0$ typically chosen to be 1. As the algorithm requires $\theta_k$ to be non-negative and bounded, a safeguarding procedure is often used (e.g.~\cite{Loreto:19}):
\begin{equation}
   \label{proc2.1}
     \theta_{k+1} = \left\{
	       \begin{array}{ll}
		\min\{\theta_{\max},\displaystyle\frac{1}{\parallel s_{k}\parallel}\}     & \mathrm{if \ } s_{k}^Ty_{k} \le 0 \\
		\min\{\theta_{\max},\max\{\theta_{\min},\displaystyle\frac{s_{k}^Ts_{k}}{s_{k}^Ty_{k}}\}\}     & \mathrm{if\ } s_{k}^Ty_{k}>0
		 \end{array},
	     \right.
   \end{equation}
   where $\displaystyle 0<\theta_{\min}<\theta_{\max}<\infty$.

The parameter $\beta_k$ can be chosen in several ways; in~\cite{Bm01}, the authors proposed the following three choices:
\begin{align}
   \label{formula1}
\beta_k&=\frac{(\theta_ky_k-s_k)^Tg_{k+1}}{s_k^Ty_k}\\
   \label{formula2}
\beta_k&=\frac{\theta_ky_k^Tg_{k+1}}{\alpha_k\theta_{k-1}g_k^Tg_k}\\
   \label{formula3}
\beta_k&=\frac{\theta_kg_{k+1}^Tg_{k+1}}{\alpha_k\theta_{k-1}g_k^Tg_k}
\end{align}
noting that under certain assumptions, these are equivalent to the formula introduced by Perry \cite{Perry:78} and its further modifications by Polak–Ribi\`ere  and Fletcher–Reeves, respectively. With any of these choices, it is possible that the search direction $d_{k+1}$ computed by \eqref{E:searchdir} may fail to be a descent direction; in this case, the authors proposed setting $d_{k+1} = -\theta_k g_{k+1}$ to ``restart'' the algorithm.

In their paper, Birgin \& Mart\'{\i}nez studied the performance of several variants of this approach, including the three choices for $\beta_k$, using $\theta_k = 1$ instead of the spectral step (corresponding to the classical conjugate gradient approach), and a heuristic for the initial choice of stepsize $\alpha$ when performing the Wolfe line search. Ultimately it was determined that the Perry formula~\eqref{formula1} using the spectral step and the heuristic for Wolfe line search gave the best performance.


This work presents an extension of the SCG method that allows for $g_k$ to be a subgradient at points where $f$ is non-differentiable. This algorithm can also be interpreted as a generalization of the spectral subgradient method ~\cite{Loreto:19}, which corresponds to the special case where $\beta_k = 0$ in equation~\eqref{E:searchdir}. Additionally, we investigate two different line search strategies: the nonmonotone line search used by the spectral subgradient method, and the Wolfe line search used in \cite{Bm01}. We describe these two choices of line search in the next section.

\subsection{Line Search Strategies}
\label{sec:2.2}
As mentioned previously, the proposed spectral conjugate subgradient can use one of two line search strategies: nonmonotone, or Wolfe line searches.
\subsubsection{Nonmonotone Line Search}
\label{sec:2.2.1}
Historically, the spectral step has been paired with the nonmonotone line search described below, based on Raydan's work \cite{Raydan:97}. The reason for this pairing is the nonmonotone behavior of the step. So, it is natural to adjust the line search scheme with the nonmonotone globalization strategy of Grippo {\it et al.} \cite{GriLaLu:86} combined with the proposed globalization scheme of La Cruz {\it et al.} \cite{wmr2:05}.
\begin{align} \label{glb1}
f(x_k+\alpha d_k) &\leq  \max_{0 \leq j \leq \min\{k,M\}} f(x_{k-j}) + \gamma \alpha g_k^Td_k + \eta_k, \text{with}\\
\label{propeta}
0 < \sum_k \eta_k &= \eta < \infty.
\end{align}

{The value of $\alpha>0$ is determined by a backtracking process beginning with $\alpha=1$. For the sequence $\eta_k$, we use $\displaystyle\eta_k=\frac{\eta_0}{k^{1.1}}$, which guarantees that (\ref{propeta}) holds. This sequence helps the line search to be satisfied at early iterations. $M$ is a fixed integer with $M \geq 0$, and $0 < \gamma < 1$ is a small sufficient decrease parameter.}


The nonmonotone behavior of the line search is induced by the terms $\displaystyle\max_{0 \leq j \leq \min\{k,M\}} f(x_{k-j})$ and $\eta_k$. The parameter $M$ allows the function value to increase or remain stationary for up to $M$ iterations, which is suitable for the behavior of the spectral step, since the spectral step is associated with the eigenvalues of the Hessian at the minimum point, rather than the function values, as noted by ~\cite{fletcher:1990,glunt:1993}.
\subsubsection{Wolfe Line Search}
\label{sec:2.22}
In \cite{Bm01}, the authors use the Wolfe conditions to guarantee convergence.  Although the Wolfe conditions assume that $f$ is differentiable, we included the Wolfe line search option in our algorithm to be compared with the nonmonotone line search, since we assume $f$ is differentiable almost everywhere. Given $0<\gamma<\sigma<1$, we require that the step size $\alpha$ satisfy
\begin{equation}\label{glb2}
f(x_k + \alpha d_k) \leq f(x_k) + \gamma \alpha g_k^Td_k
\end{equation}
and
\begin{equation}
\label{glb3}
\nabla f(x_k +\alpha d_k)^Td_k\geq \sigma g_k^Td_k
\end{equation}
in every iteration, where $g_k$ is a subgradient of $f$ at $x_k$.

\subsection{Spectral Conjugate Subgradient Algorithm}
\label{sec:2.3}

Recall that a subgradient of $f$ at a point $x$ is any vector $g \in \Real^n$ that satisfies the inequality 
\begin{equation*}f(y)\geq f(x) +g^T(y-x) ,~~\forall y \in \Real^n. 
\end{equation*}
The set of all subgradients at the point $x$ is given by the subdifferential of $f$ at $x$, denoted $\partial f (x)$; a function is subdifferentiable if this set is nonempty everywhere. A necessary condition at a minimizer of $f$ is for $0 \in \partial f(x)$. Using the definitions of $s_k$ and $y_k$ from Section~\ref{sec:2.1}, but where $g_k$ is now a subgradient, we define the spectral conjugate subgradient (SCS) algorithm as follows:\\

\noindent
{\bf Algorithm SCS:}\\
\noindent
{Let $x_0\in \Real^n$ be given, and let $g_0$ represent a subgradient at $f(x_0)$. We define $d_0=-g_0$, and set $k=0$ and $MAXITER$ as the maximum allowable number of iterations. Let $M$ be an integer with $M \geq 0$, let $0<\theta_{\min}<\theta_{\max}<\infty$, let $\eta_{0}=\max(f(x_0),\parallel g_0\parallel)$, and let $0<\gamma<\sigma<1$. Then:}\\

\noindent
Repeat until $k=MAXITER$

\begin{enumerate}
\item Compute $\alpha$ based on the line search, either nonmonotone (\ref{glb1}), or Wolfe ((\ref{glb2}) and (\ref{glb3}))
\item Define $\alpha_k=\alpha$ and $x_{k+1}=x_k+\alpha_kd_k$
\item Compute $\theta_k$ by (\ref{proc2.1}) and $\beta_k$ by (\ref{formula1}),(\ref{formula2}), or (\ref{formula3})
\item Define $d = -\theta_kg_{k+1} +\beta_ks_k$
{\item If $d^Tg_{k+1}\leq -10^{-3}\parallel d \parallel_2 \parallel g_{k+1}\parallel_2$
\subitem then $d_{k+1}=d$
\subitem else $d_{k+1}=-\theta_kg_{k+1}$}
\item $k=k+1$
\end{enumerate}

\noindent
{\bf Remarks:}
\begin{itemize}
\item Step 5 is the heuristic suggested by Birgin \& Mart\'{\i}nez to guarantee $d$ is a descent direction. If the angle between $d$ and $-g_{k+1}$ lies outside of $[-\frac{\pi}{2}, \frac{\pi}{2}]$ (to some small tolerance), the algorithm is restarted to use the direction of the negative subgradient.
\item For values of $x_{k+1}$ where $f$ is differentiable, the subgradient $g_{k+1}$ is equal to the gradient $\nabla f(x_{k+1})$.
\item When $\beta =0$ and the nonmonotone line search (\ref{glb1}) is used, the SCS algorithm becomes the spectral subgradient method described in \cite{Loreto:19}.
\item When using the Perry formula~\eqref{formula1} for $\beta_k$, it is possible in nonsmooth problems for the denominator $s_k^Ty_k$ to be zero. In this case we simply set $\beta_k = 0$, equivalent to the spectral subgradient method.
\item For the Wolfe line search ((\ref{glb2}) and (\ref{glb3})), we adopted the implementation by Taviani \cite{Taviani}
\item {The algorithm terminates upon reaching the maximum iteration count ($MAXITER$)}, and the best value $f_{\min}$ is reported, along with the corresponding point $x$.
\end{itemize}

\section{CT problem with Total Variation}
\label{S:CT}

In computed tomography (CT) imaging, a two- or three-dimensional image of a compactly supported function is recovered from a discrete sampling of the so-called X-ray transform of the function. This transform consists of integrals along every possible line through the function; in the two-dimensional case, which we consider in this article, it is equivalent to the better-known Radon transform. The simplest type of two-dimensional discretization, known as parallel-beam geometry, is obtained by equal spacing of both the affine and angular parameters in the X-ray transform, producing a set of measurements known as a sinogram. Figure~\ref{F:CTimg1} shows an example of a simple two-dimensional object (the Shepp-Logan phantom) and its sinogram.

Discretizing both the domain of the function and the X-ray transform results in a linear system of equations:

\begin{equation}
Ax + \eta = b,
\end{equation}
where $x\in\Real^{n}$ is the non-negative image being reconstructed, $b \in \Real^m$ is the sinogram, and $A$ is the $m \times n$ system matrix representing the geometry of the CT system. The vector $\eta$ represents measurement noise that arises naturally during the real-life CT imaging process. When $\Vert \eta \Vert$ is small, and the X-ray transform of the object is well-sampled, the image can be reconstructed by solving the non-negative least-squares problem:
\begin{equation}
\min_{x\in\Real^{n}_+}~~\frac{1}{2} \Vert Ax - b \Vert^2, \label{E:lsq}
\end{equation}
typically using an iterative method, due to the large size of the matrices involved. In more challenging circumstances, one must instead solve a regularized problem:
\begin{equation}
\min_{x\in\Real^{n}_+}~~\frac{1}{2} \Vert Ax - b \Vert^2 + \mu \phi(x),\label{E:lsqreg}
\end{equation}
where $\phi: \Real^n \to \Real$ is an appropriate regularization function, and $\mu > 0$ is a weighting parameter. Two scenarios of this type are if $\Vert \eta \Vert$ is large, or if the angular sampling of the X-ray transform is insufficient to recover the high-frequency components of $x$. Both of these situations arise as a result of trying to reduce dose to a patient in medical CT imaging; the former is often called low-dose imaging, and the latter is sparse-view imaging. In both cases, solving the unregularized problem \eqref{E:lsq} results in poor quality images, as shown in Figure~\ref{F:CTimg1}.

\begin{figure}
\caption{CT imaging example. Left: true $400 \times 400$ pixel digital Shepp-Logan phantom. Center-left: parallel-beam sinogram corresponding to 360 views taken over 180$^\circ$; the affine parameter is the $y$-axis and the angular parameter the $x$-axis. Center-right: low-dose image reconstructed using unregularized least-squares with 20\% Gaussian noise added to $b$. Right: Sparse-view mage reconstructed using unregularized least squares with only 60 views taken over 180$^\circ$ }\label{F:CTimg1}
\includegraphics[width=\linewidth]{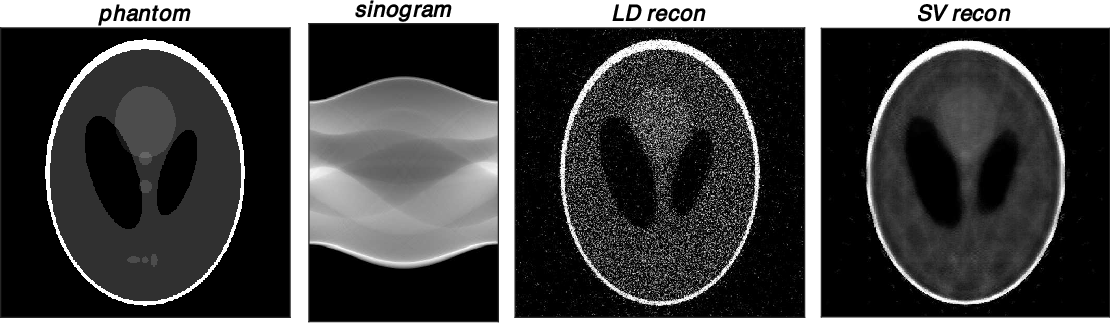}
\end{figure}

A common choice of regularizer $\phi(x)$ is the isotropic total variation (TV) function:

\begin{equation}
\phi_{TV}(x) = \sum_{1 \leq i,j \leq N-1} \sqrt{\left(x_{i,j+1}-x_{i,j}\right)^2 +  \left(x_{i+1,j}-x_{i,j}\right)^2},\label{E:TV}
\end{equation}
in which $x$ is reshaped as an $N \times N$ image (with $N^2 = n$), indexed by row and column. A routine computation gives the partial derivatives of $\phi_{TV}$ as:
\begin{align}
\frac{\partial}{\partial x_{i,j}} \phi_{TV} (x) &= \frac{2x_{i,j} - x_{i+1,j} - x_{i,j+1}}{\sqrt{ \left(x_{i+1,j}-x_{i,j}\right)^2 + \left(x_{i,j+1}-x_{i,j}\right)^2}} \label{E:dTV}\\
&~~~+ \frac{x_{i,j} - x_{i,j-1}}{\sqrt{(x_{i,j} - x_{i,j-1})^2 + (x_{i+1,j-1} - x_{i,j-1})^2}}\notag \\
&~~+  \frac{x_{i,j} - x_{i-1,j}}{\sqrt{(x_{i,j} - x_{i-1,j})^2 + (x_{i-1,j+1} - x_{i-1,j})^2}}.\notag
\end{align}
Clearly the derivative is undefined whenever the denominator of any of the three terms in \eqref{E:dTV} is zero. This situation arises frequently as it occurs when pixels have the same values as their neighbors; furthermore, such images are likely to be minimizers of~\eqref{E:lsqreg} as they correspond to images with small TV. Note, however, that the numerator of the respective term must also be zero whenever the denominator is zero. Following~\cite{ND10}, we therefore define a subgradient by simply setting any term in \eqref{E:dTV} with a denominator of zero equal to zero.

\section{Convergence Analysis}\label{S:convergence}

In Appendix~\ref{A:proof} we provide a general proof of convergence for a subgradient descent algorithm using the nonmonotone line search. The proof parallels that of~\cite{Griva} (Theorem 11.7) for the smooth case with Armijo conditions. The proof requires the following hypotheses on the objective function $f$:
\begin{enumerate}
\item $f$ is convex and continuously differentiable almost everywhere, but not necessarily at minimizers.
\item $f$ is bounded below, and the level set \[\Omega(x_0)=\{x\in \Real^n\ \mid f(x)\leq f(x_0)\}\] is compact, for any starting point $x_0\in \Real^n$. 
\item The subgradient $g$ is Lipschitz continuous on $\Omega(x_0)$ ; namely, there exists $L > 0$ such that
\begin{equation*}\parallel g_x-g_y \parallel \leq L \parallel x-y\parallel, \;\;\forall x,y\in \Omega(x_0),~ g_x \in \partial f(x), g_y \in \partial f(y)
\end{equation*}
\item The norm of the subgradients is bounded, namely, there exists a constant $G>0$ such that $\parallel g \parallel \leq G$ for all $g \in
\partial f(x)$ and $x \in \Omega(x_0)$. Note that this occurs as a consequence of (2) and (3) (see \cite{zzl06}).\end{enumerate}

Additionally, the proof requires that the search directions $d_k$ generated by the algorithm satisfy the following properties, where $g_k \in \partial f(x_k)$:
\begin{enumerate}
\itemsep1em
\item  $\displaystyle \frac{d_k^T g_k}{\Vert d_k \Vert \Vert g_k \Vert} \leq -\epsilon$, where $\epsilon > 0$, and
\item There exist constants $\mu,\nu \geq 0$ such that $\mu \Vert g_k \Vert \leq \Vert d_k \Vert \leq \nu$ for all $k$.
\end{enumerate}
The first condition ensures that $d_k$ is not too close to being perpendicular to the negative subgradient, and is guaranteed to hold in our case, as a result of the heuristic in Step 5 of Algorithm SCS. The second condition ensures that the norm of $d_k$ is bounded above, and also not too small relative to the norm of $g_k$. This second condition straightforwardly holds in the case when $d_k = -\theta_{k-1}g_k$, but is difficult to prove when $d_k = -\theta_{k-1}g_k + \beta_{k-1} s_{k-1}$, especially given that there are three different choices of parameter $\beta_k$. 

In fact we believe it may not hold in some cases; in particular, in the numerical experiments conducted in Section~\ref{sec:5.1}, we found that while the norm of $d_k$ was generally not too large relative to the norm of $g_k$, it was occasionally much smaller; in some cases, by up to 10 orders of magnitude. We therefore experimented with a modified version of the algorithm which rescales $d_k$ in such cases. Specifically, given specified values of $\mu$ and $\nu$, we replaced Steps 4 and 5 of algorithm SCS with the following procedure, which rescales the direction accordingly:
\begin{enumerate}
\setcounter{enumi}{3}
\item Define $d = -\theta_kg_{k+1} +\beta_ks_k$
{\item If $d^Tg_{k+1} > -10^{-3}\parallel d \parallel_2 \parallel g_{k+1}\parallel_2$
\subitem then $d_{k+1}=-\theta_kg_{k+1}$}\\
{else if $\Vert d \Vert > \nu$
\subitem then $d_{k+1} = \nu \frac{d}{\Vert d \Vert}$}\\
{else if $\Vert d \Vert < \mu \Vert g_k \Vert$
\subitem then $d_{k+1} = \mu \Vert g_k \Vert \frac{d}{\Vert d \Vert}$} \\
{else
\subitem then $d_{k+1} = d$}
\end{enumerate}
In our numerical experiments with this modified version, however, we did not see any improvement in the performance of the algorithm, even in those cases where the norm of $d_k$ became very small. We therefore consider only the version of Algorithm SCS specified in Section~\ref{sec:2.3}. We note that the absence of a convergence guarantee does not preclude the algorithm from being useful in practice; for example, it is well-known that quasi-Newton methods for minimization of smooth functions tend to also work well on nonsmooth problems, despite a lack of convergence theory; see for example~\cite{LO13}.

\section{Numerical Results}
\label{S:experiments}

{We provide numerical results that demonstrate the performance of the proposed method using a collection of widely recognized nonsmooth benchmark problems frequently employed for evaluation purposes. Additionally, we showcase an application in CT reconstruction with the utilization of total variation, as outlined in Section~\ref{S:CT}}.

\subsection{Nonsmooth test problems}
\label{sec:5.1}

{The collection of problems consists of 10 nonsmooth minimization problems, taken from~\cite{Napsu07} and detailed in Table \ref{table:tst14}. This table provides information such as the optimal value $f_* = f(x_*)$ and the dimension $n$. For further details on this problem set, including suggested initial points, please refer to \cite{Napsu07}.}

\begin{table}[htbp]
\centering
\caption{Properties of the ten nonsmooth minimization problems.}
\label{table:tst14}
\footnotesize
\begin{tabular}{|c|c|c|}\hline
Prob                 & $f_*$             & n   \\\hline \hline
 P1: MAXQ            & 0.0               & 20             \\
 P2: MXHILB          & 0.0               & 50      \\
 P3: Chained LQ      & $-(n-1)2^{1/2}$    & 2              \\
 P4: Chained CB3 I   & $2(n-1)$          & 20             \\
 P5: Chained CB3 II  & $2(n-1)$          & 20              \\
 P6: Activefaces     & 0.0               & 2             \\
 P7: Brown 2         & 0.0               & 2              \\
 P8: Mifflin 2       & -34.795           & 50              \\
 P9: Crescent I      & 0.0               & 2                \\
 P10:Crescent II     & 0.0               & 2                \\ \hline
\end{tabular}

\end{table}

We are comparing a total of 8 versions of the spectral conjugate subgradient method, corresponding to two choices of line search (nonmonotone or Wolfe) and four different approaches to calculate the parameter $\beta$ in \eqref{E:searchdir}. We let $\beta_0=0$ correspond to the original spectral subgradient method, and $\beta_1$, $\beta_2$ and $\beta_3$ correspond to the formulas (\ref{formula1}), (\ref{formula2}), and (\ref{formula3}) respectively. Each algorithm was run for $MAXITER = 1000$ iterations, and the the best function value, called $f_{\min}$, was reported.

{We evaluate the performance of each algorithm by analyzing the results using performance profiles~\cite{Dolan02}}. A detailed explanation of how we adopted the profiles is given in \cite{Loreto:19}.
{Considering a set of solvers $S$, which in this scenario includes the 8 variations of the aforementioned algorithm, and the set of problems $P$ given in Table~\ref{table:tst14}, the performance ratio for each solver-problem pair is defined as follows:}

\[r_{p,s}=\frac{t_{p,s}}{\min\{t_{p,s}:s \in S\}},\]
where {the performance measure $t_{p,s}$ represents a specific metric, e.g. computation time required  by solver $s$ to solve problem $p$.} The performance of solver $s$ is then assessed by computing a function
\begin{equation*}
\rho_s(\tau)=\frac{1}{n_p}size\{p\in P:r_{p,s} \leq \tau \}.
\end{equation*}
{Here  $n_p$ denotes the total number of problems in the set $P$, and $\tau$ is a parameter ranging between 1 and $r_M$. The max value $r_M$ is determined to satisfy $r_M\geq r_{p,s}$ for all $p$ in the set $P$ and $s$ in the set $S$.} This function is then plotted for all solvers $s$ to produce the performance profiles; the best-performing algorithms have profiles in the upper left of the plot, and worse-performing algorithms in the lower right. Typically a log scale is used for $\tau$ for ease of visualization; {then the value of $\rho_s(\tau)$ at zero (when $\tau = 1$) provides the percentage of problems in which solver $s$ achieved the best performance.}

{In our experiments, we focus on three different performance metrics: the relative error between $f_{\min}$ and $f_{*}$ (if $f_{*}$ is non-zero), or the absolute error if $f_{*}$ is zero; the total number of function evaluations needed to obtain $f_{\min}$; and total CPU time consumed by the algorithm. An algorithm is considered to solve a problem $p$ if the error between $f_{\min}$ and $f_{*}$  is less than $10^{-1}$. If solver $s$ fails to solve problem $p$, $r_{p,s}$ is set equal to $r_M$.}

To represent the 8 methods in the performance profiles, we use the notation $WB_i$ for the Wolfe line search and the four choices of $\beta$, ($i=1,\ldots 4$). Likewise, we set the notation $NMB_i$ for the nonmonotone line search.

The performance profile using error as the performance measure is shown in Figure~\ref{fig:figure1}. The plot shows the algorithm with Wolfe line search and parameter $\beta = 0$, denoted $WB_0$, as the most precise solver overall since it solves accurately $50\% $ of the problems, and $70\%$ of them if a small error is tolerated. However, $WB_0$ didn't solve the other $30\%$ of the problems based on our criteria. This makes $WB_0$ the most accurate algorithm but the least robust.

On the other hand, $NMB_2$ solves $100\%$ of the problems when some error is tolerated, though only $20\%$ with maximum accuracy. This makes $NMB_2$ the most robust algorithm. It is worth mentioning $WB_2$ and $WB_3$ as good alternatives to combine accuracy and robustness. $WB_2$ solves accurately $50\%$ of the problems and $90\%$ of them when some error is tolerated, similar to $WB_3$ except this algorithm is slightly less accurate. We note that $NMB_0$, originally presented at \cite{Loreto:19}, is outperformed by both $WB_0$ and $NMB_2$ in terms of accuracy and robustness, respectively.

\begin{figure*}[ht!]
\caption{Performance profile based on error between $f_{\min}$ and $f_{*}$ using the nonmonotone line search with M=7, and the Wolfe line search per each parameter $\beta$.}
\label{fig:figure1}       
\includegraphics[width=0.85\linewidth]{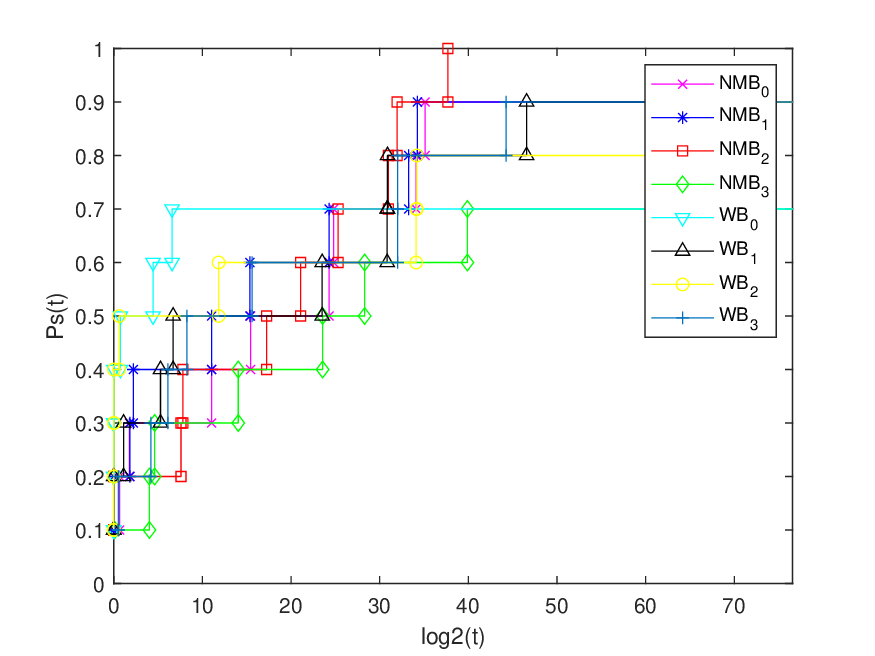}
\end{figure*}
The performance profiles provide a useful framework for comparing different algorithms based on different criteria. When analyzing performance in terms of errors, though, knowing the magnitude of the errors can provide additional relevant information since, in practical cases, being more accurate for a problem in relative terms may not be meaningful the errors are already quite small. Table \ref{table:tb2} shows the minimum error achieved by each algorithm for each problem.

\begin{table}[ht!]
\caption{Error between $f_{\min}$ and $f_{*}$ using the nonmonotone line search with M=7, and the Wolfe line search per each parameter $\beta$}
\label{table:tb2}
\footnotesize
\centering
\footnotesize
\begin{tabular}{|c|c|c|c|c|c|c|c|c|} \hline
Prob & $NM\beta_0$ & $NM\beta_1$ & $NM\beta_2$ & $NM\beta_3$ &$W\beta_0$ &$W\beta_1$& $W\beta_2$& $W\beta_3$\\ \hline
P1 & 0.000e+00 & 0.000e+00 & 4.244e-06 & 3.753e-06 & 9.625e-21 & 1.191e-08 & 3.642e-04 & 3.798e-08 \\
P2 & 3.269e-01 & 3.890e-01 & 7.147e-02 & 4.352e-01 & 7.549e-01 & 7.565e-01 & 7.259e-01 & 8.492e-01 \\
P3 & 4.530e-03 & 4.530e-03 & 9.178e-03 & 7.290e-02 & 2.201e-10 & 2.201e-10 & 2.224e-10 & 4.012e-09 \\
P4 & 8.887e-03 & 3.026e-10 & 6.660e-04 & 3.741e-03 & 1.179e-01 & 3.619e-03 & 1.098e-06 & 1.514e-05 \\
P5 & 5.551e-02 & 5.780e-02 & 1.627e-02 & 1.772e-02 & 2.729e-02 & 3.507e-02 & 2.485e-02 & 2.354e-02 \\
P6 & 8.220e-03 & 7.770e-03 & 2.880e-02 & 3.124e-03 & 5.290e-01 & 1.918e-05 & 1.848e-07 & 5.639e-05 \\
P7 & 2.645e-02 & 1.462e-02 & 2.983e-03 & 1.367e-01 & 7.105e-13 & 1.406e-03 & 1.325e-02 & 3.117e-03 \\
P8 & 2.645e-02 & 1.462e-02 & 2.983e-03 & 1.367e-01 & 1.420e-12 & 2.739e-03 & 2.777e-02 & 6.156e-03 \\
P9 & 1.254e-04 & 4.262e-04 & 1.765e-02 & 2.252e-03 & 1.979e-03 & 3.554e-03 & 9.177e-05 & 6.270e-03 \\
P10 & 2.843e-04 & 2.844e-04 & 3.007e-05 & 2.190e-06 & 1.347e-07 & 1.347e-07 & 1.361e-07 & 1.915e-07 \\ \hline
\end{tabular}
\end{table}

\begin{figure*}[ht!]
\caption{Performance profile based on the number of function evaluations using the nonmonotone line search with M=7, and the Wolfe line search per each parameter $\beta$}.
\label{fig:figure2}
\includegraphics[width=0.85\linewidth]{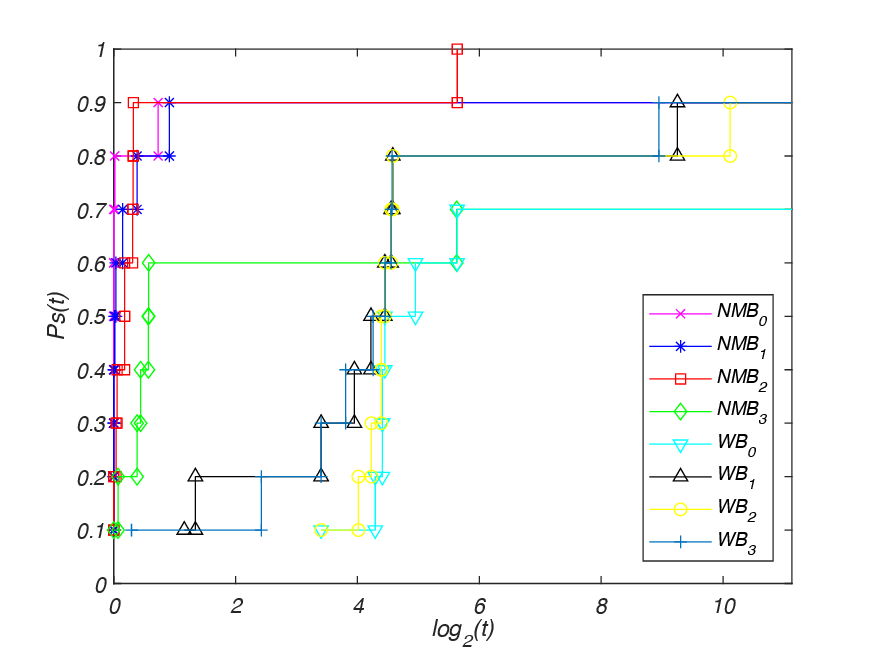}
\end{figure*}

Figure \ref{fig:figure2} shows the performance profile based on the total number of function evaluations needed to find $f_{min}$. It shows the spectral subgradient method $NMB_0$\cite{Loreto:19} as the most efficient algorithm. $NMB_0$ solves $80\%$ of the problems with the fewest function evaluations. Similarly, $NMB_1$ presents as an efficient option, slightly worse than $NMB_0$. However, $NMB_0$ and $NMB_1$ were unable to solve one problem to within the specified threshold.

On the other hand, $NMB_2$ solves $100\%$ of the problems within the specified threshold while requiring only somewhat more function evaluations than $NMB_0$ and $NMB_1$. This would make the trade-off robustness versus efficiency worthy, making $NMB_2$ the best alternative in our opinion.

Figure \ref{fig:figure2} also shows that the nonmonotone line search considerably outperforms the Wolfe line search with respect to the number of function evaluations required. Indeed, solvers using Wolfe line search are clearly on the bottom right side of the chart, meaning they solve the fewest problems at the highest cost. It is worth mentioning that Figure \ref{fig:figure1} showed $WB_0$ as the most precise solver and $WB_2$ as a suitable alternative in terms of precision. But Figure \ref{fig:figure2} shows them both as the most expensive in terms of function evaluations, which makes these options unattractive.

\begin{figure*}[ht!]
\caption{Performance profile based on CPU using the nonmonotone line search with M=7, and the Wolfe line search per each parameter $\beta$}
\label{fig:figure3}       
\includegraphics[width=0.85\linewidth]{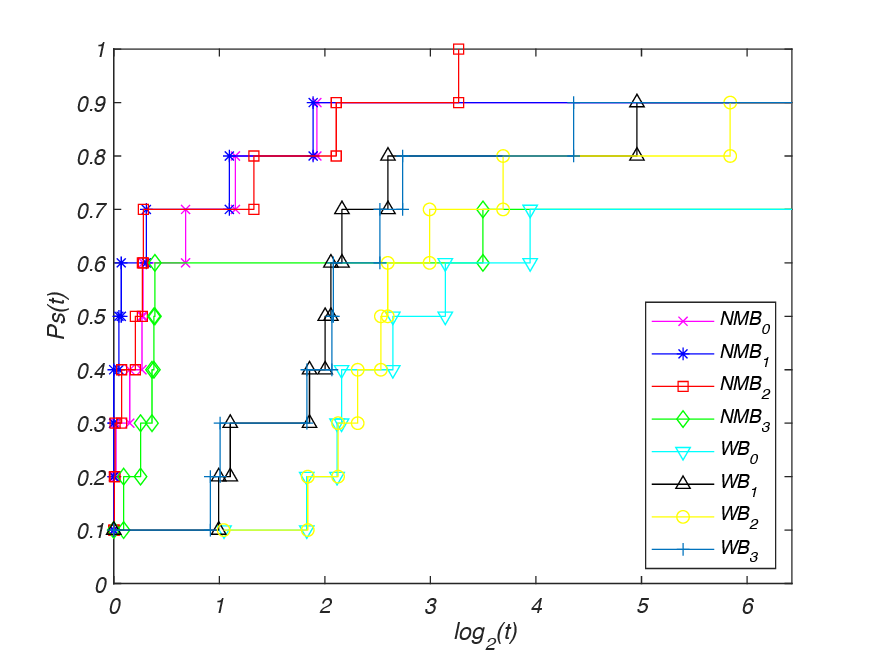}
\end{figure*}

With respect to CPU time, Figure \ref{fig:figure3} shows quite similar behavior when compared to the profile based on function evaluations (Figure~\ref{fig:figure2}). {This outcome is anticipated since evaluating the function is the primary factor that drives the computational effort of these algorithms.} $NMB_2$ seems to be the best alternative, being more robust than $NMB_0$ and $NMB_1$, given it solves $100\%$ of the problems, although with slightly worse performance.

After analyzing all performance profiles, we can conclude that $NMB_2$ is the best alternative when factoring in all the criteria. This is based on its high robustness, as it was the only method that solved all ten problems to within the specified threshold. $NMB_2$ also provides solutions with accuracy comparable with the rest of the solvers, although clearly outperformed by $WB_0$ on this measure. And, $NMB_2$ is in the group of top performing solvers based on computation time, with only slightly poorer performance than $NMB_0$ and $NMB_1$.

{In conclusion, we can state that the} Wolfe line search tends to improve the accuracy of the algorithms, but clearly at the expense of many more function evaluations, and hence computational effort.

\subsection{CT problem with total variation}
\label{sec:5.2}

In this section, we study the effect of the choice of parameter $\beta$ when applying Algorithm SCS to a set of TV-regularized CT reconstruction problems. All experiments were run using the AIRTools II Matlab toolbox~\cite{HJ18} to generate parallel-beam CT system matrices $A$ and the corresponding sinograms $b$. An image size of $400 \times 400$ pixels was used, resulting in a problem size of $n = 160,000$. We then applied Algorithm SCS to the regularized least squares problem \eqref{E:lsqreg}, using TV \eqref{E:TV} as the regularizer, whose subgradient was computed using \eqref{E:dTV} and the subsequently discussed method.

A total of 45 reconstruction problems were considered, corresponding to:
\begin{itemize}
\item Three different images (phantoms) from the AIRTools II toolbox: {\tt shepplogan}, {\tt threephases}, and {\tt grains}. All three phantoms are piecewise constant, making them well-suited to total variation regularization. They are also all elements of the $n$-dimensional box $[0,1]^n \subset \Real^n$.
\item Five different imaging scenarios: three low-dose scenarios with Gaussian noise of 1\%, 5\% and 10\% added to $b$, and two sparse-view scenarios using  60 and 30 angular views. In the three low-dose scenarios, a total of 360 views and 566 lines through the object were simulated, giving a sinogram size of $m = 203,760$; this is reduced by a factor of 6 and 12, respectively, for the two sparse-view scenarios. The sparse-view scenarios did not include noise.
\item Three different values of $\mu$, controlling the weighting of the regularization term in \eqref{E:lsqreg}. For the three low-dose scenarios, these were $\mu = 25$, $\mu = 250$, and $\mu = 2500$, while for the sparse-view scenarios, we used $\mu = 0.5$, $\mu = 5$, and $\mu = 50$. These values provide different weighting on the nonsmooth (TV) component of the function, and were found to give good results on at least one imaging scenario; in general, more challenging problems (high noise or small number of views) required a larger value of $\mu$ to regularize the problem effectively. Note that in the sparse-view scenarios, the size of the first term in \eqref{E:lsqreg} is smaller than in the low-dose scenarios (since the sinogram size is smaller), which is why a smaller value of $\mu$ can be used.
\end{itemize}

Additionally, because the desired solution to each problem is a vector $x \in [0,1]^n$, we modified Algorithm SCS to include projection onto this set (denoted $\Omega$). This is accomplished as follows: before every iteration of Step 1, redefine $d_k$ as
\begin{equation}
d_k = P_{\Omega} (x_k + d_k) - x_k,
\end{equation}
where $P_{\Omega}(x) = \min\{\max\{x,0\},1\}$ is the projection onto $\Omega$. This guarantees that $x_{k} \in \Omega$ for all $k$, provided that the step size $\alpha_k$ does not exceed 1. As this is not the case for the Wolfe line search, in this section we only use the nonmonotone line search with $M=7$. We, therefore, consider only four algorithms, namely NM$\beta_0$ to NM$\beta_3$ from the previous section. We note that the use of Wolfe line search is prohibitively expensive for this problem in any event, due to the high computational cost of function evaluations. Algorithm SCS was run for 200 total iterations for every case, or until $\Vert g_k \Vert \leq 10^{-10} $; though in practice, the algorithm ran for the maximum number of 200 iterations in every instance. While this may indicate that the algorithm is being stopped before convergence is achieved, we note that the image quality shown later in Figure~\ref{F:CTimages_final} is satisfactory, for the most part. The initial point used was $x_0 = 0 \in \Real^n$. 

Figure~\ref{F:CTprofileFbest} shows the performance profile based on the lowest objective function value, $f_{\min}$ found by each solver for each problem. Note that this differs from the measure used for the nonsmooth test problems (relative or absolute difference between $f_{\min}$ and the true optimum $f_*$), since $f_*$ is unknown for these problems. A solver was considered to have solved the problem only if $f_{\min}$ was within 10\% of the lowest value found. We observe that the choice of $\beta_2$ provides the best overall results, solving 80\% of the problems to a high degree of accuracy; furthermore, it was the only solver to solve all problems within the 10\% threshold. The choice of $\beta_1$ was the next best-performing, solving about 85\% of the problems to within the specified threshold, but generally to a lower accuracy than when using $\beta_2$. The choice of $\beta_3$ and $\beta_0$ (the spectral subgradient method) performed poorly in comparison, solving only about 50\% and 33\% of the problems to the specified tolerance, respectively.

\begin{figure*}[ht!]
\caption{Performance profile for CT reconstruction problems, based on lowest objective function value $f_{\min}$ found by Algorithm SCS using nonmonotone line search with each choice of $\beta$.}
\label{F:CTprofileFbest}       
\includegraphics[width=0.85\linewidth]{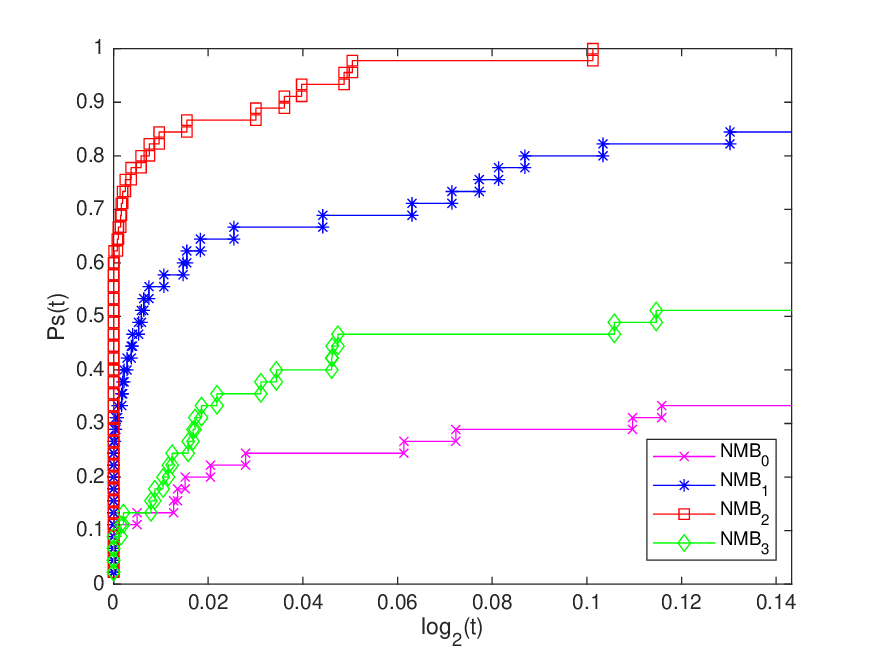}
\end{figure*}

Figure~\ref{F:CTprofileFevals} shows the performance profile generated for the four solvers, using the total number of function evaluations as the performance measure. The profile based on CPU time is omitted, as the results were essentially the same. Unlike for the nonsmooth test problems (Figure~\ref{fig:figure2}), we observe that the SCS algorithm using $\beta_2$ required significantly fewer function evaluations than when using $\beta_1$ or $\beta_3$. As a total of 200 iterations was run for each solver, this indicates that the direction chosen as a result of using $\beta_2$ generally resulted in less backtracking than the other methods. We note that the spectral subgradient method ($\beta_0$) similarly required few function evaluations when successful; however, as noted in the previous paragraph, it failed to solve two thirds of the problems to the acceptable threshold.

\begin{figure*}[ht!]
\caption{Performance profile for CT reconstruction problems, based on number of function evaluations used by Algorithm SCS using nonmonotone line search with each choice of $\beta$. }
\label{F:CTprofileFevals}       
\includegraphics[width=0.85\linewidth]{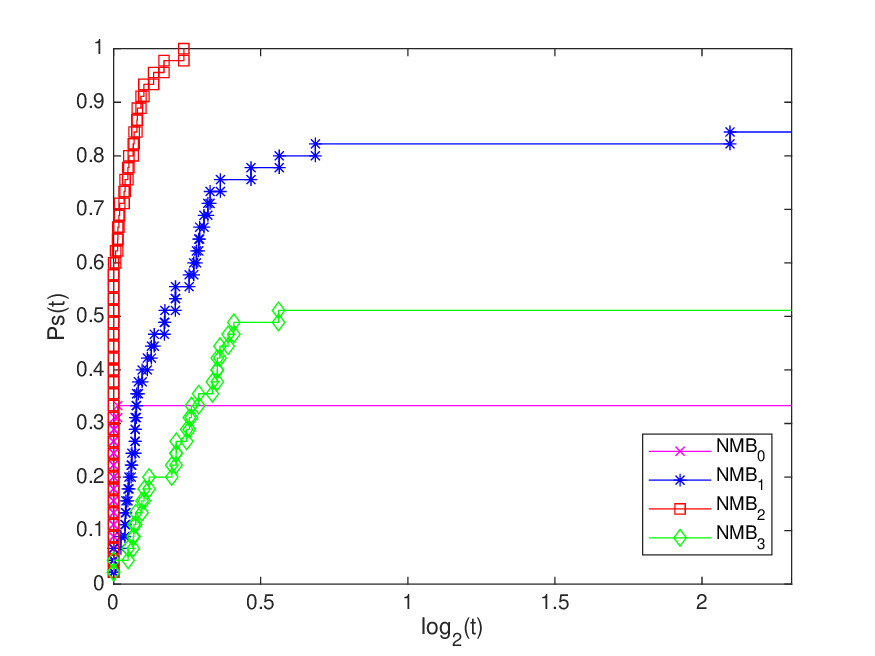}
\end{figure*}

Figure~\ref{F:CTimages_final} shows the best images reconstructed by Algorithm SCS using $\beta_2$. We can observe that the reconstructed images are generally of high quality, although there are still artifacts in some of the images corresponding to more challenging scenarios, such as low-dose imaging with 10\% Gaussian noise (final column). We note that this is a particularly challenging scenario, as we wished to test the performance of the algorithm under a wide variety of conditions; a noise level closer to 1\% is often used (see e.g.~\cite{JJHJ12}). Fine-tuning the choice of the weighting parameter, $\mu$, could also improve image quality in these cases, although it is difficult to determine a parameter that works well in all scenarios; for example, the choice of $\mu = 2500$ gave the best results for the 10\% Gaussian noise scenario for the {\tt threephases} and {\tt grains} phantoms, but poorer results for the {\tt shepplogan} phantom than when $\mu = 250$ was used.

\begin{figure*}[ht!]
\caption{Best images reconstructed by Algorithm SCS using $\beta_2$. The first column shows the true phantom images used to generate the data, while the remaining five columns show images reconstructed for the five scenarios (sparse-view imaging with 60 and 30 views, and low-dose imaging with 1\%, 5\% and 10\% Gaussian noise, respectively). The regularization parameter $\mu$ giving the best results is shown.}\label{F:CTimages_final}
\includegraphics[width=\linewidth]{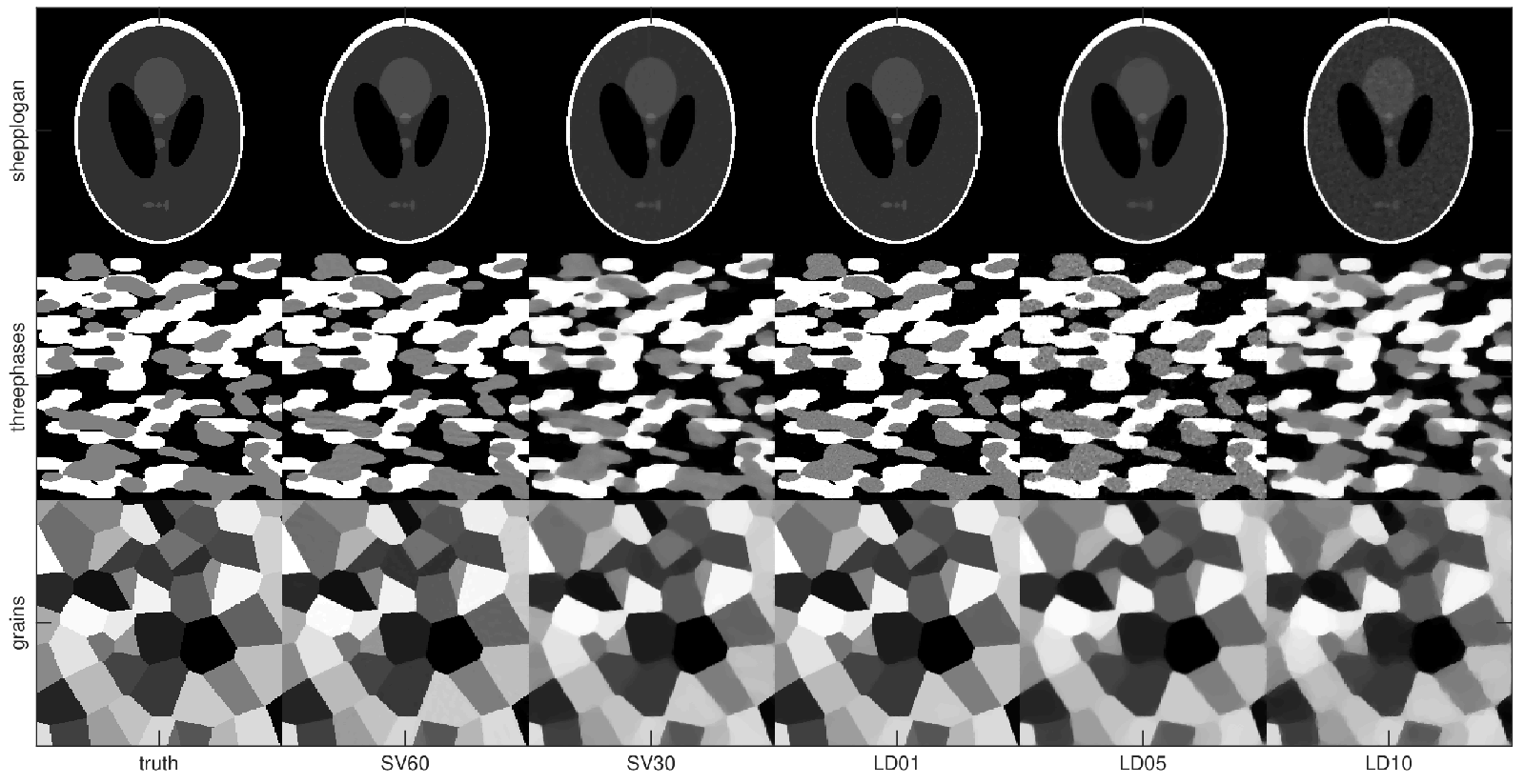}
\end{figure*}

The performance profile in Figure~\ref{F:CTprofileFbest} is generated based on the minimal value of the function $f$ given in (\ref{E:lsqreg}). While this is a sensible way to evaluate the performance of the algorithm in the context of optimization, for an imaging application, the quality of the image produced is also of interest. Two common metrics used to assess the quality of a reconstructed image $x$ versus a ground truth image $y$ are peak signal to noise ratio (PSNR) and Structural Similarity index (SSIM)\cite{SSIM}. The PSNR is given by
\begin{align}
PSNR(x,y) &= 10 \log_{10} \left( \frac{y_{max}}{MSE(x,y)} \right), \textrm{where} \label{E:PSNR} \\MSE(x,y) &= \frac{1}{N}\sum_j \left(x_{j} - y_{j}\right)^2, \: \textrm{ and } y_{max} = \max_j \{ y_j \}\notag, 
\end{align}
with higher PSNR corresponding to better agreement between the two images. SSIM is computed using the formula
\begin{align}
SSIM(x,y) &= \frac{ (2 \mu_x \mu_y + C_1)(2 \sigma_{xy} + C_2)}{(\mu_x^2 + \mu_y^2 + C_1)(\sigma_x^2 + \sigma_y^2 + C_2)},\label{E:SSIM}
\end{align}
where $\mu_x, \mu_y$ are the means of the two images, and $\sigma_x, \sigma_y$ and $\sigma_{xy}$ the variances and covariance, respectively; the constants $C_1$ and $C_2$ are included to avoid instability when the denominator is small. The SSIM produces a score between 0 and 1, with 1 corresponding to identical images. 

In Table~\ref{T:PSNR} we show the best SSIM and PSNR values obtained by each choice of $\beta$ for each of the 15 imaging scenarios. The value of $\mu$ giving the best values of PSNR and SSIM was used in each case. The reason for this is that if the value of $\mu$ is chosen to be too large or too small, then an algorithm that does the best job of minimizing $f(x)$ might in fact end up producing an image of lower quality; for example, by oversmoothing the image if $\mu$ is large. The table confirms our earlier observation that the choice of $\beta_2$ is the best-performing one, although with respect to these metrics, using $\beta_1$ provides generally comparable results. The images produced using $\beta_3$ had very low PSNR and SSIM scores in several cases, bringing the average performance of the algorithm down, while $\beta_0$ provided the poorest results overall.

\begin{table}
\caption{Best PSNR and SSIM values attained by each algorithm for each imaging scenario. The overall mean and standard deviation values for each column are shown at the bottom.}\label{T:PSNR}
\footnotesize
\begin{tabular}{ll|rr|rr|rr|rr}
		&&$\beta_0$		&&$\beta_1$		&&$\beta_2$		&&$\beta_3$	\\
		&&PSNR	&SSIM	&PSNR	&SSIM	&PSNR	&SSIM	&PSNR	&SSIM\\
  \hline
LD01	&shepplogan	&28.76	&0.978	&42.64	&0.989	&43.36	&0.996	&41.08	&0.996\\
	&threephases	&27.64	&0.901	&31.45	&0.934	&31.97	&0.950	&31.25	&0.914\\
	&grains	       &30.42	&0.873	&32.32	&0.968	&32.22	&0.975	&30.83	&0.952\\
LD05	&shepplogan	&29.04	&0.840	&31.51	&0.978	&30.97	&0.981	&29.02	&0.907\\
	&threephases	&19.32	&0.670	&22.82	&0.759	&22.73	&0.770	&22.43	&0.728\\
	&grains	       &24.74	&0.754	&26.49	&0.909	&26.29	&0.906	&25.01	&0.608\\
LD10	&shepplogan	&21.40	&0.844	&30.11	&0.961	&30.45	&0.933	&26.49	&0.918\\
	&threephases	&17.88	&0.459	&18.76	&0.657	&18.51	&0.636	&17.74	&0.470 \\
	&grains	       &18.73	&0.743	&24.49	&0.862	&25.35	&0.873	&21.64	&0.363\\
SV60	&shepplogan	&33.00	&0.968	&42.60	&0.998	&40.24	&0.996	&43.27	&0.998\\
	&threephases	&22.13	&0.762	&29.15	&0.885	&30.20	&0.969	&29.68	&0.961\\
	&grains	       &26.68	&0.832	&30.07	&0.945	&36.07	&0.988	&39.17	&0.993\\
SV30	&shepplogan	&29.93	&0.926	&36.15	&0.991	&37.73	&0.990	&32.86	&0.978\\
	&threephases	&18.19	&0.608	&21.88	&0.802	&22.30	&0.805	&20.53	&0.737\\
	&grains	       &23.34	&0.728	&29.43	&0.945	&29.08	&0.938	&30.15	&0.939\\
\hline									
mean                &&	24.75	&0.792            &29.99 	&0.906 &30.50 &0.914 &29.41 	&0.857  \\
stdev               &&$\pm$ 4.83 &$\pm$ 0.136	 &$\pm$ 6.60 &$\pm$ 0.095	 &$\pm$ 6.67	 &$\pm$ 0.100	&$\pm$ 7.27 &$\pm$ 0.178
\end{tabular}
\end{table}

\section{Final Remarks}
\label{S:conclusions}
In this article, we develop a new Spectral Conjugate Subgradient (SCS) method for nonsmooth unconstrained optimization, by combining conjugate directions proposed by Birgin \& Mart\'{\i}nez~\cite{Bm01} and the spectral subgradient method~\cite{Loreto:19}. We investigate the use of two different line search strategies and several choices of the parameter $\beta$ used to determine conjugate directions. We present the results of numerical experiments with standard nonsmooth test problems, as well as problems from CT imaging using total variation regularizations. {The results are examined utilizing performance profiles, which enable the comparison of various approaches in terms of precision and computational cost.}

The numerical results show that the combination of the nonmonotone line search with the parameter $\beta_2$ (analogous to the Polak–Ribi\`ere formula), which we denote by $NMB_2$, was the most successful approach. This is based on its high robustness, given that it solved all the problems to within a specified threshold, for both the nonsmooth test problems and the CT imaging problems. Additionally, the computational effort required by $NMB_2$ was on par with the least expensive approaches in both sets of experiments.

It is interesting to note that in Birgin \& Mart\'{\i}nez' original article they found that the choice of $\beta_1$ (Perry's formula) was the best option, while the best option in our experiments was $\beta_2$. It is difficult to compare the results directly, given that their experiments were only on smooth problems, using Wolfe line search. However, a possible explanation for why the choice of $\beta_2$ works better in our experiments is that the formula for $\beta_1$ has a denominator of $s_k^Ty_k$, which is more often equal to zero when solving nonsmooth problems than for differentiable problems. Indeed, in our numerical experimentation we noticed this term is often zero for some problems. In this case, we set $\beta_1 = 0$, defaulting to the spectral subgradient approach. On the other hand, when the problem is differentiable, the Wolfe line search ensures that $s_k^Ty_k > 0$, avoiding this issue.


As discussed in Section~\ref{S:convergence} and Appendix~\ref{A:proof}, we have established sufficient conditions for the algorithm to converge, which may not always be satisfied by the implementation of SCS described in this paper. Future work could involve investigating variations of the method for which rigorous guarantees of convergence can be established.

%

\section*{Disclosure statement}

{The authors do not have any significant financial or non-financial interests to declare.}

\section*{Funding}

{The National Science Foundation provided support for this work (grant DMS-2150511).}

\section*{Notes on contributor(s)}

M. Loreto and T. Humphries co-authored the manuscript. C. Raghavan, K. Wu and S. Kwak participated in a Research Experience for Undergraduates (REU) funded by the above-mentioned NSF Grant in Summer 2022, providing initial implementation of the spectral conjugate subgradient method and preliminary experimentation.

\section*{Data Availability Statement}
{The authors can provide the MATLAB code and data used in this work upon request.}

%
%
%
%
%

\section{References}

\bibliographystyle{plain}
\bibliography{hybridOpt.bib}

\appendix

\section{General proof of convergence for nonmonotone line search with subdifferentiable functions}\label{A:proof}

We first prove two lemmas concerning subdifferentiable functions that are needed for the proof.

\begin{lemma}\label{L:lemma1} Let $f: \mathbb{R}^n \to \mathbb{R}$ be subdifferentiable, and let $x, p \in \mathbb{R}^n$. Define $F:\mathbb{R} \to \mathbb{R}$ as $F(t) = f(x+tp)$. Then $g^Tp \in \partial{F}(t)$ for any $g \in \partial f (x+tp)$
\end{lemma}
\begin{proof}Let $g \in \partial f(x + t p)$. Then by definition,
\begin{align*}
F(t) - F(s) = f(x + tp) - f(x+sp) \geq g^T(tp - sp) = g^Tp (t -s),
\end{align*}
and so $g^Tp \in \partial F(t)$. 
\end{proof}
Note: This is a special case of Theorem 3.51 from~\cite{Nguyen}.

\begin{lemma}\label{L:lemma2}
Suppose $f: \mathbb{R}^n \to \mathbb{R}$ is subdifferentiable and its subgradient is Lipschitz continuous with constant $L$, and let $g \in \partial f(x)$. Then for all $x,p \in \mathbb{R}^n$,
\begin{align*}
\vert f(x+p) - f(x) - g^T p \vert \leq \frac{L}{2} \Vert p \Vert^2
\end{align*}
\end{lemma}
\begin{proof}  Define $F(t)$ as in Lemma~\ref{L:lemma1}. Following Proposition 1.6.1 from~\cite{Niculescu}, we have for $F:\Real \to \Real$:
\begin{align*}
F(1) - F(0) = \int_0^1  \phi(t) \: dt
\end{align*}
where $\phi$ is a function such that $\phi(t) \in \partial F (t)$ for all $t \in (0,1)$. By Lemma~\ref{L:lemma1}, we can define $\phi(t) = h^T p$ where $h \in \partial f(x+tp)$. Therefore (following the proof of Lemma 4.1.12 from [2], for differentiable functions), we have
\begin{align*}
f(x+p) - f(x) - g^Tp &= \left(\int_0^1 h^Tp \: dt\right) - g^Tp \\
&= \int_0^1 (h-g)^T p \: dt \\
\vert f(x+p) - f(x) - g^T p \vert &\leq \int_0^1 \Vert h - g \Vert \Vert p \Vert \: dt, ~\text{by the Cauchy-Schwarz inequality.}
\end{align*}
Since $h \in \partial f(x+tp)$, $g \in \partial f(x)$, Lipschitz continuity then gives $\Vert h - g \Vert \leq L \Vert tp \Vert$, so 
\begin{align*}
\Vert f(x+p) - f(x) - g^T p \Vert &\leq L \Vert p \Vert^2 \int_0^1 \: t dt\\
&\leq \frac{L}{2} \Vert p \Vert^2
\end{align*}
\end{proof}
We now prove two properties related to the sequence generated by the nonmonotone line search. Given $g_k \in \partial f(x_k)$ and a search direction $d_k$, the nonmonotone line search chooses the smallest value $h_k \in \{0,1,2,\dots\}$ such that $\alpha_k = \sigma^{h_k}$ satisfies:
\begin{align}
f(x_k + \alpha_k d_k ) \leq \max_{0 \leq j \leq M} f(x_{k-j}) + \gamma \alpha_k g_k^T d_k + \eta_k,
\end{align}
where $M \geq 0$, $0 < \sigma < 1$, $0 < \gamma < 1$, and $\sum_{k=0}^\infty \eta_k < \infty$. The formula above differs slightly from (\ref{glb1}), since we are considering limiting behavior and therefore ignore the case $k < M$. For convenience, we will also assume $\sigma = \frac{1}{2}$ in the proofs below; other choices of $\sigma$ still lead to convergence, but introduce inconvenient constants. 

\begin{lemma}\label{L:lemma3} Suppose $f: \mathbb{R}^n \to \mathbb{R}$ is subdifferentiable and its subgradient is Lipschitz continuous with constant $L$. Under the nonmonotone line search condition, either $\alpha_k~=~1$, or we have the following lower bound on $\alpha_k$:
\begin{align*}
\alpha_k &\geq \frac{\gamma-1}{L} \frac{g_k^T d_k}{\Vert d_k \Vert^2}
\end{align*}
\end{lemma}
\begin{proof} Note that $\alpha_k = 1$ when $h_k = 0$, so we consider only the case $h_k \geq 1$, meaning $\alpha_k < 1$. There are two possibilities:
\begin{enumerate}
\item $\displaystyle \max_{0 \leq j \leq M} f(x_{k-j}) = f(x_k)$

In this case, since $\alpha_k < 1$, the line search condition must have been violated with step length $2 \alpha_k$:
\begin{align}
f(x_k + 2 \alpha_k d_k) - f(x_{k}) > 2\gamma \alpha_k g_k^T d_k + \eta_k.\label{E:l3e1}
\end{align}

By Lemma~\ref{L:lemma2}, we also have:
\begin{align}
f(x_k + 2 \alpha_k d_k) - f(x_{k}) -2  \alpha_k g_k^T d_k &\leq \frac{1}{2} L \Vert 2 \alpha_k d_k \Vert^2 \notag\\
f(x_{k}) - f(x_k + 2 \alpha_k d_k)  &\geq -2  \alpha_k g_k^T d_k - 2 \alpha_k^2 L \Vert d_k \Vert^2 -\eta_k,\label{E:l3e2}
\end{align}
where the absolute value can be removed on the left side since the quantity is positive, and $\eta_k$ can be subtracted on the second line because it is also positive. Adding (\ref{E:l3e1}) and (\ref{E:l3e2}) then gives:
\begin{align*}
0 &\geq 2\alpha_k (\gamma - 1)g_k^T d_k - 2 \alpha_k^2 L\Vert d_k \Vert ^2   \\
\alpha_k L\Vert d_k \Vert ^2 &\geq (\gamma - 1) g_k^T d_k \\
\alpha_k &\geq \frac{\gamma-1}{L} \frac{g_k^T d_k}{\Vert d_k \Vert^2},
\end{align*}
giving the desired lower bound. We now consider the second possibility:
\item  $\displaystyle \max_{0 \leq j \leq M} f(x_{k-j}) = f(x_{k-J})$, $1 \leq J \leq M$.

In this case, $\alpha_k$ is the first step size satisfying 
\begin{align*}
f(x_k + \alpha_k d_k ) \leq  f(x_{k-J}) + \gamma \alpha_k g_k^T d_k + \eta_k,
\end{align*}

Define $\alpha_k^*$ to be the step size that would be chosen in the case  $\displaystyle \max_{0 \leq j \leq M} f(x_{k-j}) = f(x_k)$, which was analyzed above. Suppose $\alpha_k < \alpha_k^*$, or equivalently, $\alpha_k^* = 2^n \alpha_k$ for some integer $n \geq 1$. Then it must be the case that
\begin{align*}
f(x_k + \alpha_k^* d_k) &> f(x_{k-J}) + \gamma \alpha_k^* g_k^T d_k + \eta_k
\end{align*}
since otherwise $\alpha_k^*$ would have been accepted as the step length, as it precedes $\alpha_k$ in the sequence $\{2^{-h_k}$\}. But since $f(x_{k-J}) \geq f(x_k)$, this implies
\begin{align*}
f(x_k + \alpha_k^* d_k) &> f(x_{k}) + \gamma \alpha_k^* g_k^T d_k + \eta_k,
\end{align*}
contradicting the definition of $\alpha_k^*$. Thus in this second case we must have $\displaystyle \alpha_k \geq \alpha^* \geq \frac{\gamma-1}{L} \frac{g_k^T d_k}{\Vert d_k \Vert^2}$ as well. Intuitively, the step size chosen by a nonmonotone line search with $M \geq 1$ can never be smaller than that chosen by the monotone (Armijo) line search corresponding to $M=0$, since the former is a less restrictive condition.
\end{enumerate}
\end{proof}

\begin{lemma}\label{L:lemma4} Let $\{f(x_k)\}_{k=0}^\infty$ be the sequence generated by the nonmonotone line search procedure from some $x_0 \in \Real^n$, and suppose further that $f(x) \geq C$ for all $x \in \Omega$, where $\Omega = \{x \mid f(x) \leq f(x_0) \}$ is compact. Then
\begin{align*}
 \sum_{k=0}^\infty \left[  \max_{0 \leq j \leq M} f(x_{k-j}) - f(x_{k+1})\right] < \infty
\end{align*}
\end{lemma}
\begin{proof} We show this converges by induction on $M$. First consider the case $M=0$. Then we have a telescoping series,
\begin{align*}
 \sum_{k=0}^\infty\left[ f(x_{k}) - f(x_{k+1})\right] &= f(x_0) - \lim_{k \to \infty} f(x_k) \\
 &\leq f(x_0) - C
\end{align*}
where $C$ is the lower bound on $f(x)$; compactness of $\Omega$ is needed to ensure that the limit exists. 

Now, suppose the series converges for all $M$ less than $M^*\geq 1$. Let $\mathcal{K} \subseteq \mathbb{N}$ be the set of indices $k$ for which  $\displaystyle \max_{0 \leq j \leq M^*} f(x_{k-j}) = f(x_{k - M^*})$, and $\overline{\mathcal{K}}$ be its complement. Then we have

\begin{align*}
 \sum_{k=0}^\infty\left[  \max_{0 \leq j \leq M^*} f(x_{k-j}) - f(x_{k+1})\right] = \sum_{k \in \mathcal{K}} \biggl[ f(x_{k -M^*}) - f(x_{k+1})\biggl] +  \sum_{k\in\overline{\mathcal{K}}}\left[  \max_{0 \leq j \leq (M^*-1)} f(x_{k-j}) - f(x_{k+1})\right]
\end{align*}

The second summation is a subseries of  $\displaystyle \sum_{k=0}^{\infty}\left[  \max_{0 \leq j \leq (M^*-1)} f(x_{k-j}) - f(x_{k+1})\right]$, which converges by the induction hypothesis; therefore it must also converge. As for the first series, we have
\begin{align*}
\sum_{k \in \mathcal{K}} \biggl[ f(x_{k - M^*}) - f(x_{k+1}) \biggl] &= \sum_{k \in \mathcal{K}} \biggl[f(x_k) - f(x_{k+1}) + f(x_{k-M^*})-f(x_k) \biggl] \\
&= \sum_{k \in \mathcal{K}}\biggl[ f(x_k) - f(x_{k+1}) \biggl] +  \sum_{k \in \mathcal{K}}\biggl[ f(x_{k-M^*})-f(x_{k}) \biggl]
\end{align*}
The second summation can be reindexed by $k^* = k-1$, and letting $\mathcal{K}^* = \{k -1 \mid k \in \mathcal{K} \}$. We then have
\begin{align*}
\sum_{k \in \mathcal{K}} \biggl[ f(x_{k - M^*}) - f(x_{k+1}) \biggl] &=  \sum_{k \in \mathcal{K}}\biggl[ f(x_k) - f(x_{k+1}) \biggl] +  \sum_{k^* \in \mathcal{K^*}}\biggl[ f(x_{k^*-(M^*-1)})-f(x_{k^*+1}) \biggl] \\
&\leq \sum_{k \in \mathcal{K}} \biggl[f(x_k) - f(x_{k+1})\biggl] +  \sum_{k^* \in \mathcal{K}^*} \biggl[\max_{0 \leq j \leq (M^*-1)} f(x_{k^*-j})-f(x_{k^*+1})\biggl] 
\end{align*}
Both of these are subseries of convergent series (a telescoping series in the first case, and by the induction hypothesis in the second case), so the entire series must also converge.
\end{proof}

We can now prove the main result. This proof closely parallels that of~\cite{Griva} (Theorem 11.7) for the smooth case with Armijo conditions.

\begin{theorem} Suppose $f: \mathbb{R}^n \to \mathbb{R}$ is subdifferentiable and its subgradient is Lipschitz continuous with constant $L$. Additionally, $f(x) \geq C$ for all $x \in \Omega$, where $\Omega = \{x \mid f(x) \leq f(x_0) \}$ is compact. Finally, the sequence of search directions $d_k$ generated by the algorithm satisfy the following two conditions:
\begin{enumerate}
    \item $\mu \Vert g_k \Vert \leq \Vert d_k \Vert \leq \nu$  for some constants $\mu,\nu > 0$, where $g_k \in \partial f(x_k)$
    \item $\displaystyle \frac{d_k^T g_k}{\Vert d_k \Vert \Vert g_k \Vert} \leq -\epsilon$ for some $\epsilon > 0$
    \end{enumerate}
Then, the sequence generated by the nonmonotone line search satisfies
\begin{align*}
 \displaystyle \lim \inf_{k \to \infty} \Vert g_k \Vert = 0
\end{align*}
\end{theorem}

\begin{proof} By the nonmonotone line search condition,
\begin{align*}
\sum_{k=0}^\infty -\gamma \alpha_k g_k^T d_k &\leq \sum_{k=0}^\infty\left[  \max_{0 \leq j \leq M} f(x_{k-j}) - f(x_{k+1})\right] + \sum_{k=0}^\infty \eta_k \\
\sum_{k=0}^\infty -\gamma \alpha_k g_k^T d_k&< \infty~~\text{(By Lemma~\ref{L:lemma4} and definition of $\eta_k$)}\\
\sum_{k=0}^\infty \gamma \epsilon \alpha_k \Vert g_k \Vert \Vert d_k \Vert &< \infty~~\text{(by second condition on $d_k$)}\\
\gamma \epsilon \mu \sum_{k=0}^\infty \alpha_k \Vert g_k \Vert^2 &< \infty~~\text{(by first condition on $d_k$)}.
\end{align*}
Convergence of the series therefore implies that $\displaystyle \lim_{k \to \infty} \alpha_k \Vert g_k \Vert^2 = 0$. Suppose now that $\displaystyle \lim \inf_{k \to \infty} \Vert g_k \Vert > 0$. Then it must be the case that $\displaystyle \lim_{k \to \infty} \alpha_k = 0$. By Lemma~\ref{L:lemma3}, we have either $\alpha_k = 1$ or $\displaystyle \alpha_k \geq \frac{\gamma-1}{L} \frac{g_k^T d_k}{\Vert d_k \Vert^2}$; this second bound simplifies to  $\displaystyle \alpha_k \geq \frac{(1-\gamma) \epsilon \mu}{\nu^2 L}\Vert g_k \Vert^2$ under the provided assumptions on $d_k$. Letting $\displaystyle \Gamma = \frac{(1-\gamma) \epsilon \mu}{\nu^2 L}$, we therefore have $\alpha_k \geq \min \{1, \Gamma \Vert g_k \Vert^2 \}$, with $\Gamma > 0$. Therefore $\alpha_k \to 0$ if and only if $\displaystyle \lim \inf_{k \to \infty} \Vert g_k \Vert = 0$, a contradiction. We conclude that $ \displaystyle \lim \inf_{k \to \infty} \Vert g_k \Vert = 0$.
\end{proof}

\end{document}